\documentclass[10pt]{article}
\usepackage{cmap}
\usepackage[pdftex]{graphicx}
\usepackage{epsfig}
\usepackage{latexsym}
\usepackage{epigraph}
\usepackage{mathtools}
\usepackage{setspace}
\usepackage[russian]{datetime2}
\usepackage{enumitem}
\usepackage{amsfonts,amssymb,mathrsfs,amscd,amsmath,amsthm}
\usepackage{verbatim}
\usepackage[mathscr]{eucal}
\usepackage{fixmath}
\ifpdf
\usepackage{hyperref}
\else
\usepackage[dvipdfmx]{hyperref}
\fi
\usepackage[usenames,dvipsnames,svgnames,table]{xcolor}
\usepackage[all]{xy}

\def\thtext#1{
  \catcode`@=11
  \gdef\@thmcountersep{. #1}
  \catcode`@=12
}

\def\threst{
  \catcode`@=11
  \gdef\@thmcountersep{.}
  \catcode`@=12
}

\newcounter{nxt}

\theoremstyle{plain}
\newtheorem{thm}{Theorem}
\newtheorem*{thh}{Theorem}
\newtheorem{prop}[thm]{Proposition}
\newtheorem{ass}[thm]{Assertion}

\newtheorem{lem}[thm]{Lemms}

\theoremstyle{definition}
\newtheorem{examp}[thm]{Example}
\newtheorem{rk}[thm]{Remark}

\def\one{\mathbf{1}}
\def\bb{\boldsymbol{b}}
\def\bp{\boldsymbol{p}}
\def\bv{\boldsymbol{v}}
\def\bw{\boldsymbol{w}}
\def\bx{\boldsymbol{x}}
\def\by{\boldsymbol{y}}

\def\:{\colon}

\def\a{\alpha}
\def\b{\beta}
\def\g{\gamma}
\def\om{\omega}

\def\conv{\operatorname{conv}}

\def\hB{{\widehat{B}}}
\def\hT{{\widehat{T}}}

\newcommand{\R}{{\mathbb R}}
\newcommand{\Z}{{\mathbb Z}}

\begin{document}
\title{Minimal Fillings of Finite Metric Spaces and Convex Polyhedra}
\author{A.\,O.~Ivanov,  A.\,A.~Tuzhilin}

\abstract{Problem of minimal parametric generalized fillings of finite metric space $M$ (a version of optimal connection problem) leads to construction of a convex multidimensional polyhedra $W_G$ for each tree $G$ connecting $M$. The weight of the parametric minimal filling can be found as maximum on $W_G$ of the special linear function corresponding to the distance vector of $M$. It is proved that the union of vertex sets of the polyhedra $W_G$ over all possible $G$ forms an extremal subset, i.e., coincides with the vertex set of its convex hull.}
\date{}

\maketitle

\section{Introduction}
The problem of minimal filling of a finite metric space was posed by A.~Ivanov and A.~Tuzhilin~\cite{IT_MSb} as a generalization of Steiner Problem on shortest networks (see, for example~\cite{IT_UMN}), and of M.~Gromov minimal filling problem for Riemannian manifolds~\cite{Gromov}. Let us recall the necessary definitions and results.

Let $M$ be a finite metric space, and $G=(V,E)$ be a connected graph connecting $M$, that is, $M\subset V$. Let $\om$ be a nonnegative weight function defined on the edges of $G$. It generates a pseudometric $\rho_\om$ on the set of vertices of the graph as follows: $\rho_\om(x,y)=\min\om(\g_{xy})$, where the minimum is taken over all paths $\g_{xy}$ connecting vertices $x$ and $y$, and the weight $\om(\g_{xy})$ of a path $\g_{xy}$ is the sum of the weights of all the edges of this path. A graph $G$ with weight function $\om$ is called a \emph{filling of the metric space $M$} if for any points $m$, $m'$ in $M$ the inequality $|mm'|\le\rho_\om(m,m')$ holds, that is, the distance along $G$ is not less than the original distance in $M$. The graph $G$ is called a \emph{filling type}.

The \emph{weight of a minimal filling of a metric space $M$} is defined as $\inf \om(G)$, where the infimum is taken over all fillings of $M$, that is, over all $G$ and $\om$, and the weight $\om(G)$ of a graph $G$ is the sum of the weights of all its edges. A filling  which this infimum is achieved at is called a \emph{minimal filling of a finite metric space $M$}. If we fix the type of a filling, that is, a connected graph $G$ connecting $M$, and take the infimum only over the weight functions $\om$ that transform $G$ into a filling of $M$, we obtain\emph{ parametric minimal fillings of a given type $G$} and their \emph{weights}.

A.~Ivanov and A.~Tuzhilin~\cite{IT_MSb} proved theorems on the existence of a minimal parametric filling for any type $G$ and a minimal filling, and described a number of their properties. In particular, it was shown that as a minimal filling of the space $M$ one can always choose a tree whose vertex degrees are equal to $1$ or $3$, where the set $M$ coincides with the set of vertices of degree $1$ of this tree (such trees will be called \emph{binary}). It was also proved that the length of the shortest tree connecting a finite subset $M$ of points of the metric space $(X,\rho)$ is bounded from below by the weight of the minimal filling of the metric space $(M,\rho)$. This bound is attained, in particular, as Z.~Ovsyannikov~\cite{Ovs} noted, equality holds for $M\subset\R^n$ with the metric generated by the $\max$-norm. B.~Bednov and P.~Borodin~\cite{BB} described all Banach spaces in which the shortest tree for any finite subset is a minimal filling.

To calculate the weight of minimal fillings, it turns out to be useful to consider trees with arbitrary (not necessarily non-negative) weights; see~\cite{IOST}. Such fillings are called \emph{generalized}. It is shown in~\cite{IOST} that although the weight of a generalized minimal parametric filling of type $G$ may be less than the weight of a ordinary minimal parametric filling of the same type $G$, the weight of a minimal filling is the same as the weight of a generalized minimal filling.

A.~Ivanov and A.~Tuzhilin~\cite{IT_MSb} wrote down formulas for the weight of minimal parametric fillings for metric spaces consisting of $3$ and $4$ points. It turned out that the weight of a filling in this case is equal to half the weight of some Hamiltonian cycle in the complete graph $K(M)$ on the vertex set $M$. Later, A.~Eremin~\cite{Eremin} obtained a general formula for the weight of a minimal parametric generalized filling of type $G$ of a space $M$ in terms of so-called multi-tours of the tree $G$. A \emph{multi-tour of multiplicity $k$ of a tree $G$} is a cyclic route $\mu$ in the complete graph $K(M)$ that passes through each vertex exactly $k$ times and is consistent with $G$ in the following sense: for each partition of the set $M$ generated by an edge of $G$, exactly $2k$ edges of the route $\mu$ connect vertices from different subsets of this partition. The  \emph{semiperimeter $|\mu|$} of a multi-tour $\mu$ of multiplicity $k$ is the sum of the lengths of all its edges (taking into account multiplicities) divided by $2k$. According to A.~Eremnin's formula, the weight of a minimal parametric generalized filling of type $G$ is $\sup|\mu|$, where the least upper bound is taken over all multi-tours of $G$. A.~Eremin also showed that it is sufficient to consider only the so-called \emph{irreducible\/} multi-tours, the possible multiplicity of which was roughly estimated from above, see~\cite{Eremin}. Thus, instead of the least upper bound in the weight formula, one can take the maximum over a finite (possibly very large) set of irreducible multi-tours.

The minimum parametric filling problem can be formulated as a linear programming problem. Indeed, see~\cite{IT_MSb}, the conditions on the filling weight function are linear inequalities, and the weight of the minimum parametric filling is the minimum of the linear function under these linear constraints. As a result of A.~Ivanov and A.~Tuzhilin's application of the duality principle to this problem, see~\cite{IT_LP}, the concept of multi-tour obtained a geometric interpretation. It was observed that, upon transition to the dual linear programming problem, the set of admissible values of the variables ceases  to depend on the metric of the space $M$ and is completely determined by the structure of the tree $G$. Thus, to each binary tree $G$ connecting $M$, there is associated a convex polytope $W_G$, and the weight of the minimal parametric generalized filling of type $G$ is equal to the maximum of the linear function $f_M$ on this polytope (the coefficients of the function $f_M$ are distances in the space $M$). This maximum is obviously attained at one of the vertices of the polytope $W_G$, so the weight of the minimal parametric generalized filling is equal to the maximum of the function $f_M$ on the vertex set of the polytope $W_G$. In~\cite{IT_LP} it is shown that the rational points of the polytope $W_G$ correspond to multi-tours of the tree $G$, an estimate is obtained for the multiplicity of multi-tours corresponding to vertices that improves the estimate in~\cite{Eremin}, and explicit formulas are written out for the weight of minimal parametric generalized fillings of spaces with the number of points at most $6$. Note that for $5$-point spaces the weight formula was obtained earlier by B.~Bednov~\cite{BBed} from different considerations. O.~Shcherbakov~\cite{Shch1} noticed that the vertices of the polyhedron $W_G$ correspond exactly to irreducible multi-tours, thus the weight formulas from~\cite{Eremin} and~\cite{IT_LP} are equivalent: in the first, the maximum is taken over the set of irreducible multi-tours of the tree $G$, which are determined by combinatorial and algebraic methods, in the second --- over the set of vertices of the multidimensional convex polyhedron $W_G$, but a natural one-to-one correspondence is established between these two sets.

The complexity of a binary tree structure can be characterized by the number of so-called \emph{moustaches\/} --- pairs of boundary vertices with a common adjacent vertex. O.~Shcherbakov~\cite{Shch1} obtained a complete description of irreducible multi-tours and polytopes $W_G$ in the case where $G$ has two moustaches (the so-called ``snake'' type binary trees), as well as for trees with three moustaches (jointly with A.~Ivanov, see~\cite{IvShch} and~\cite{Shch2}). It was shown that for trees with two moustaches, all irreducible multi-tours have multiplicity $1$, and for trees with three moustaches, the multiplicity of a multi-tours does not exceed $2$. Moreover, in the latter case, every such tree with $6$ or more vertices of degree $1$ has an irreducible multi-tour of multiplicity $2$.

To find a minimum filling, one must, generally speaking, enumerate all possible types of trees $G$ connecting $M$. In~\cite{IT_LP} it was verified (using a computer) that for a fixed $n$, $n\le6$, the union of the vertex sets of all polytopes $W_G$, where the union is taken over all trees $G$ connecting an $n$-point $M$, itself forms the vertex set of some convex polyhedron. O.~Shcherbakov~\cite{Shch1} noted that this is also true for any $n$, if we restrict ourselves to all possible trees of the ``snake'' type (in this case, the vertices of each $W_G$ are vertices of the same multidimensional cube). The main result of this paper is the following.

\begin{thh}\label{thm:all_conv}
Consider the family of all binary trees connecting a finite set $M$, $|M|=n$. For each such tree $G$, construct a polyhedron $W_G$, denote by $V_G\subset\R^N$ the set of its vertices, and let $V=\cup_GV_G\subset\R^N$. Then the set $V$ is extremal, that is, $V$ is the set of vertices of its convex hull $\conv V\subset\R^N$.
\end{thh}

\section{Preliminaries}
This section presents the necessary results of the theory of minimal fillings of finite metric spaces, Graph Theory and Linear Programming.

\subsection{Graphs, Metric Spaces and Fillings}
A graph $G=(V,E)$ is a pair consisting of a finite set $V$, whose elements are called \emph{vertices}, and a set $E$ of two-element subsets of $V$, which are called \emph{edges}. An edge $e=\{u,v\}$ is said to \emph{connect\/} vertices $u$ and $v$, vertices $u$ and $v$ in this case are called \emph{adjacent}, and an edge $e$ and a vertex $u$ or $v$ are \emph{incident}. The number of vertices adjacent to a fixed vertex $v$ is called its \emph{degree}. A graph in which any two vertices are adjacent is called \emph{complete\/} and is denoted by $K(V)$ or $K_n(V)$, where $n$ is the number of vertices. A \emph{route\/} in a graph $G$ connecting vertices $u$ and $v$ of the graph is a sequence $\g$ of the form $u=v_{i_0},\ldots,v_{i_k}=v$ of pairwise adjacent vertices. The \emph{edges of the route $\g$} are edges of the form $\{v_{i_p},v_{i_{p+1}}\}$. A route is called a \emph{chain\/} if all its edges are pairwise distinct, and a \emph{path\/} if all its vertices are pairwise distinct. A chain whose first and last vertices coincide is called a \emph{cycle}.

A function $\om\:E\to\R$ defined on a set of graph edges is called a \emph{weight\/} function, and its value $\om(e)$ on an edge $e$ is the \emph{edge weight}. The sum of the weights of all edges of a graph $G$ is called the \emph{graph weight\/} and is denoted by $\om(G)$. The sum of the weights of all edges of a route (including multiplicities) is the \emph{route weight}.

A graph is called \emph{connected\/} if any two of its distinct vertices can be connected by some route. A connected graph without cycles is called a \emph{tree}. We say that a connected graph $G=(V,E)$ \emph{connects a set $M$} if $M\subset V$. In this context, the set $M$ is called a \emph{boundary\/} of $G$, and the vertices of $M$ are called the \emph{boundary vertices}. A tree connecting $M$ will be called \emph{binary\/} if the degrees of its vertices are equal to $1$ or $3$, and the set of vertices of degree $1$ coincides with $M$.

Consider an arbitrary tree $G=(V,E)$ with boundary $M$. The graph $G=\big(V,E\setminus\{e\}\big)$, obtained from $G$ by removing edge $e$, is the union of two disjoint trees $G_1$ and $G_2$. The subsets $M_k=G_k\cup M$, $k=1,\,2$, are disjoint and form a partition of the set $M$, $M=M_1\sqcup M_2$, which we denote by $p(e)$.

A set $M$ on which a non-negative, symmetric function $M\times M\to\R$, $(m,m')\mapsto|mm'|$, is defined, satisfying the triangle inequality $|mm'|+|m'm''|\ge|mm''|$, is called a \emph{pseudometric space}, and the function itself is called a \emph{pseudometric on $M$}. If $|mm'|=0$, if and only if $m=m'|$, then we speak of a \emph{metric\/} and a \emph{metric space}.

\begin{examp}
Consider a connected graph $G=(V,E)$ with a nonnegative weight function $\om\:E\to[0,\infty)$. Then the function $d_\om\:V\times V\to\R$, defined by the formula $\rho_\om(u,v)=\min\om(\g_{uv})$, where the minimum is taken over all possible paths in $G$ connecting $u$ and $v$ if $u\neq v$, and $\rho_\om(u,u)=0$, is a pseudometric on $V$.
\end{examp}

Let $M$ be a finite metric space, $G$ be a connected graph connecting $M$, and $\om\:E\to[0,\infty)$ be a nonnegative weight function. A graph $G$ with weight function $\om$ is called a \emph{filling of the space $M$} if $|mm'|\le\rho_\om(m,m')$. The graph $G$ is called the \emph{type\/} of the filling.

\begin{examp}\label{example:net_fill}
Consider an arbitrary metric space $X$ with metric $\rho$, a finite subset $M$ of it, and a connected graph $G$ connecting $M$, and let $V\subset X$. Then, on the edges of $G$, a weight function $\om_\rho\:\{v,v'\}\mapsto\rho(v,v')$ is defined naturally. The graph $G$ with weight function $\om_\rho$ is a filling of the metric space $(M,\rho)$.
\end{examp}

The \emph{weight of a minimal filling of a metric space $M$} is defined as $\inf \om(G)$, where the infimum is taken over all fillings of $M$, that is, over all connected graphs $G$ connecting $M$ and over all weight functions $\om$ that transform $G$ into a filling of $M$. A filling at which this infimum is attained is called a \emph{minimal filling}. If we fix the type of filling, that is, a connected graph $G$ connecting $M$, and take the greatest lower bound only over the weight functions $\om$ that turn $G$ into a filling of $M$, the resulting value is called the \emph{weight of the minimal parametric filling of type $G$} of $M$, and if this infimum is attained at the weight function $\om_0$, then the graph $G$ with this weight function is called the \emph{minimal parametric filling of type $G$} of $M$.

As mentioned in Introduction, these concepts were introduced in the paper~\cite{IT_MSb}, which proved the existence theorem for a minimal parametric filling of any type and the existence theorem for a minimal filling for an arbitrary finite metric space. A connection was also established between the problems of minimal filling and shortest trees (Steiner problem). As we have already seen in example~\ref{example:net_fill}, every connected graph connecting a finite subset $M$ of a metric space $(X,\rho)$ with vertices in $X$ defines a filling of $(M,\rho)$; in particular, a shortest tree for $M$ in $X$ is a candidate for a minimal filling of $M$. Consider now a finite metric space $M$ and an arbitrary isometric embedding $f\:M\to X$ of it into an arbitrary metric space $X$. Let $s(M,f,X)$ denote the length of a shortest tree for a set $f(M)$ in a space $X$. In~\cite{IT_MSb} it is shown that the weight of the minimal filling of $M$ is equal to $\inf\big\{r\in\R: s(M,f,X)\le r\big\}$, where the greatest lower bound is taken over all non-negative numbers $r$ for which there exists an isometric embedding $f\:M\to X$ in some metric space $X$ and the inequality $s(M,f,X)\le r$ holds. Therefore, the weight of the minimal filling for $M$ provides a lower bound for the length of a shortest tree with boundary $M$. Moreover, this bound is attained, for example, on finite subsets of $\R^n$ with the metric generated by the $\max$-norm.

It is also shown in~\cite{IT_MSb} that, to find a minimal filling, it is sufficient to consider only fillings whose types are essentially binary trees. Furthermore, if arbitrary trees are allowed, then it is sufficient to consider only strictly positive weight functions.

\subsection{Generalized Fillings, Multi-tours and Weight Formula}
To study fillings whose types are trees, it turns out to be useful to allow the weight function to take arbitrary values, see~\cite{IOST}. Specifically, let $M$ be a finite metric space, $G$ a connected graph connecting $M$, and $\om\:E\to\R$ an arbitrary weight function. A graph $G$ with weight function $\om$ is called a \emph{generalized\/} filling of a space $M$ of type $G$ if $|mm'|\le\rho_\om(m,m')$. Repeating the definitions from the previous section verbatim, we obtain the concepts of a \emph{minimal parametric generalized filling\/} and a \emph{minimal generalized filling\/} and their \emph{weights}. In the work~\cite{IOST} it was shown that although the weight of a minimal parametric generalized filling of a fixed type may be less than the weight of a minimal parametric filling of the same type, the weights of the minimal filling and the generalized minimal filling coincide, and a generalized minimal filling is also a minimal filling.

The transition to generalized fillings allowed us to obtain a general formula for the weight of a minimal filling; see~\cite{Eremin}. Let us recall the necessary definitions. A \emph{multi-cyclic order of multiplicity $k$\/} on a set $M$ is a mapping $\pi\:\Z_{nk}\to M$ such that
\begin{enumerate}
\item $\pi(j)\ne\pi(j+1)$ for any $j\in\Z_{nk}$;
\item The $\pi$-preimage of each element consists of exactly $k$ elements.
\end{enumerate}

Every multicyclic order $\pi$ generates a cyclic route $c_\pi=\pi(0),\pi(1),\ldots,\pi(nk-1),\pi(0)$ in the complete graph $K_n(M)$, consisting of $nk$ edges. By definition, the route $c_\pi$ traverses each vertex of the graph $K_n(M)$ exactly $k$ times. Let the numbering of points in $M$ be fixed, $M=\{m_1,\ldots,m_n\}$. Consider the linear space $\R^N$, $N=n(n-1)/2$, and numerate the components of vectors in it by pairs $\{i,j\}$, $1\le i,\, j\le n$, that is, in fact, by edges $\{m_i,m_j\}$ of the complete graph $K_n(M)$. We construct the vector $\bw^{\pi}=(w^{12},\ldots, w^{(n-1)n})\in \R^N$ by setting the component $w^{ij}$ equal to the number of occurrences of the edge $\{m_i,m_j\}$ in the route $c_\pi$. The resulting vector $\bw^{\pi}$ is called the \emph{vector of the multi-cyclic order $\pi$}.

A multi-cyclic order of multiplicity $k$ of a set $M$ is called a \emph{multi-tourl of a binary tree $G$ with boundary $M$} if for each of its edges $e$, there exist exactly $2k$ elements $p\in\Z_{nk}$ such that $\pi(p)$ and $\pi(p+1)$ lie in different components of the partition $p(e)$.

\emph{The semiperimeter of a multi-tour $\pi$} is defined as
$$
|\pi|_M=\frac1{2k}\sum_{j\in\Z_{nk}}\big|\pi(j)\pi(j+1)\big|,
$$
where $k$ is the multiplicity of the multi-tour $\pi$. A.~Eremin showed that the weight of a minimal parametric generalized filling can be calculated in terms of semiperimeters as follows, see~\cite[Theorem~4.1]{Eremin}.

\begin{ass}\label{ass:Eremin}
Let $G$ be a binary tree connecting a finite metric space $M$. The weight of the minimal parametric generalized filling of type $G$ of $M$ is equal to
$
\max_\pi|\pi|_M,
$
where the maximum is taken over all possible multi-tours of $G$.
\end{ass}

As above, denote by $\g_{xy}$ the unique path in $G$ connecting its vertices $x$ and $y$. Each multi-tour $\sigma$ generates a family of paths $\g_{\sigma(p)\sigma(p+1)}$, $p\in\Z_{nk}$, in $G$. We call such paths \emph{boundary paths of the multi tour $\sigma$}. The following result holds; see~\cite[Proposition 2.4]{Eremin}.

\begin{prop}\label{prop:mult_tour_edges}
Let $G$ be a binary tree with boundary $M$, $|M|=n$. A multi-cyclic order $\pi$ of multiplicity $k$ on $M$ is a multi-tour of $G$ if and only if for each of its edges $e$, for exactly $2k$ values $p\in\Z_{nk}$, the boundary path $\g_{\pi(p),\pi(p+1)}$ passes through $e$.
\end{prop}

Proposition~\ref{prop:mult_tour_edges} implies that multi-tours can be defined by suitable sets of boundary paths. Denote by $\Gamma(\pi)$ the family of boundary paths corresponding to the multi-tour $\pi$ (taking into account possible multiplicities). Let the numbering of points in $M$ be fixed, $M=\{m_1,\ldots,m_n\}$. Then the component $w^{ij}$ of the multi-tour vector $\bw^\pi$ is equal to the number of occurrences in the family $\Gamma(\pi)$ of the boundary path connecting $m_i$ and $m_j$.

We will also need the following statement~\cite[Proposition~2.5]{Eremin}.

\begin{prop}\label{prop:mult_tour_edges_pairs}
Let $\pi$ be a multi-tour of multiplicity $k$ of a binary tree $G$, and $v$ be a vertex of degree three in $G$. Denote by $e_1$, $e_2$, and $e_3$ the edges incident to $v$. Then for every pair $\{e_1,e_2\}$, $\{e_1,e_3\}$, $\{e_2,e_3\}$ of edges there are exactly $k$ boundary paths in $\Gamma(\pi)$ passing through it. In particular, if vertices $m$ and $m'$ of degree $1$ form moustaches of $G$, then $\Gamma(\pi)$ contains exactly $k$ boundary paths connecting $m$ and $m'$.
\end{prop}

Notice that different multi-tours can have the same multi-tours vectors. For example, if $\pi\:\Z_{nk}\to M$ is a multi-tour, then for each $p\in\Z_{nk}$, the mapping $\pi_{+p}\:\Z_{nk}\to M$, defined by the formula $\pi_{+p}(i)=\pi(i+p)$, defines a multi-tour with the same multi-tour vector. We say that $\pi_{+p}$ is obtained from $\pi$ \emph{by shifting the origin by $p$}.

Two multi-tours $\a$ and $\b$ are called \emph{equivalent\/} if $\bw^\a=\bw^\b$. The equivalence class of a multi-tour $\pi$ is denoted by $[\pi]$. Equivalent multi-tours obviously have the same multiplicity and equal semiperimeters. Using vector addition and multiplication, we can define similar operations on equivalence classes of multi-tours as follows. Let $\a$ and $\b$ be multi-tours of multiplicities $k$ and $l$, respectively. The \emph{sum\/} $[\a]+[\b]$ is the equivalence class of any multi-tour $\pi$ of multiplicity $k+l$ for which the equality $\bw^\pi=\bw^{\a}+\bw^{\b}$ holds. As such a multi-tour $\pi$, we can take the multi-tour obtained by ``sequential passing'' of the multi-tours $\a$ and $\b$. More formally, shifting the origins of the multi-tours if necessary, we replace $\a$ and $\b$ with equivalent ones so that $\a(0)=\b(0)$, after which we set $\pi(i)=\a(i)$ for $0\le i< nk$ and $\pi(i)=\b(i-nk)$ for $nk\le i<n(k+l)$. Multiplication by a positive integer is defined inductively via addition: $1[\pi]=[\pi]$, $2[\pi]=[\pi]+[\pi]$, $m[\pi]=[\pi]+(m-1)[\pi]$. Obviously, when multiplying a multi-tour by $m$, its multiplicity is also multiplied by $m$.

In the formula from the Assertion~\ref{ass:Eremin}, the maximum is taken over the infinite set of all multi-tours of the tree $G$. It turns out that it suffices to consider a subset of the so-called irreducible multi-tours. A multi-tours $\mu$ is called \emph{irreducible\/} if the equality $m[\mu]=[\pi_1]+[\pi_2]$ for some multi-tours $\pi_1$ and $\pi_2$ and a positive integer $m$ implies that $[\pi_1]=k[\mu]$ and $[\pi_2]=(m-k)[\mu]$ for some $k$, $k\le m$. In the paper~\cite{Eremin}, it is shown that the set of irreducible multi-tours of a fixed tree $G$ is finite, and an upper bound is given for their possible number.

\subsection{Generalized Fillings, Linear Programming, and Tree's Polyhedron}
The problem of finding a minimal parametric generalized filling reduces~\cite{IT_MSb} to a linear programming problem. Indeed, let, as above, $M=\{m_1,\ldots,m_n\}$ be a finite metric space whose points are numbered, and let a binary tree $G=(V,E)$ connecting $M$ be fixed. The conditions on the weight function $\om\:E\to\R$, which transforms the tree $G$ into a generalized filling of the space $M$, have the form of a system of $N=n(n-1)/2$ linear inequalities: for each pair $\{i,j\}$, $1\le i,\,j\le n$, the inequality
$$
\sum_{e\in\g_{m_im_j}}\om(e)\ge|m_im_j|,
$$
must hold where, as above, $\g_{m_im_j}$ denotes the unique path in the tree $G$ connecting its vertices $m_i$ and $m_j$. To find the minimal parametric generalized filling it is required to minimize the linear function $\sum_{e\in E}\om(e)$, where the summation is taken over all edges of the tree $G$, under these conditions. We obtain the so-called \emph{General Linear Programming Problem}, see, for example,~\cite{VasIv}. Note that the variables are the edge weights, which can be viewed as components of a $(2n-3)$-dimensional vector.
 
A.~Ivanov and A.~Tuzhilin~\cite{IT_LP} proposed to consider an equivalent dual problem of linear programming. They obtained a \emph{canonical linear programming problem\/} on the space $\R^N$, that is, the set of admissible values of variables is defined by a system of linear equations and the condition of non-negativity of all variables, see~\cite{VasIv}. We numerate the components of the vector of variables $\bx\in\R^N$ by pairs $\{i,j\}$, $1\le i,\,j\le n$, which we write for brevity as $ij$. The matrix $A_G$ of the system of equations has size $(2n-3)\times N$ and is the \emph{cut matrix\/} of the complete graph $K_n(M)$ generated by the edges of the tree $G$. Namely, the partition $p(e)=M_1\sqcup M_2$ of the set $M$ generated by the edge $e$ of the tree $G$ defines a cut of the graph $K_n(M)$, and the element $a_{ij}^e$ of the matrix $A_G$ is equal to $1$ if the vertices $m_i$ and $m_j$ of the graph $K_n(M)$ lie in different sets $M_1$ and $M_2$, and is equal to $0$ otherwise. The set $W_G$ of admissible values of the dual problem is defined by the system of equations and inequalities
$$
A_G\bx=\one_{2n-3},\qquad \bx\ge0,
$$
where $\one_k\in\R^k$ denotes the column vector of all ones. The dual problem is to maximize the function $g_M(\bx)=\sum_{ij}x^{ij}|m_im_j|$ on $W_G$. Note that $W_G$ is independent of the metric on $M$ and is completely determined by the tree $G$ connecting $M$, while the function $g_M$, on the contrary, depends only on the metric on $M$. In the paper~\cite{IT_LP} it is shown that the matrix $A_G$ is non-degenerated, and $W_G$ is a non-empty convex polyhedron, so to find the maximum it suffices to enumerate the values of the function $g_M$ at its vertices. In~\cite{IT_LP} it is shown that the rational points of $W_G$ correspond to multi-tours of the tree $G$, and the following result holds.

\begin{prop}\label{prop:m_tour_vec}
For any multi-tour $\sigma$ of multiplicity $k$ of a binary tree $G$, the equality $A_G\bw^\sigma=2k\one_{2n-3}$ holds.
\end{prop}

Also in~\cite{IT_LP}, an estimate for the multiplicity of multi-tours corresponding to the vertices of the polytope $W_G$ was obtained, improving the estimate in~\cite{Eremin}.

We will need a well-known result describing all the vertices of the polytopes of admissible values of the variables in a canonical linear programming problem; see, for example~\cite{VasIv}. Recall that a point $\bx$ is called a \emph{vertex}, or an \emph{extremal point}, or an \emph{corner point\/} of a closed subset $X\subset \R^n$ if $\bx$ does not lie inside any interval with endpoints in $X$, that is, more formally, from the existence of points $\bx_0$ and $\bx_1$ in $X$ and a real $t\in(0,1)$ for which the equality $\bx=(1-t)\bx_0+t\bx_1$ holds, it follows that $\bx_0=\bx_1=\bx$.

\begin{prop}\label{prop:Lin_Prog_Vert}
Let $A$ be a non-degenerated $m\times n$ matrix, $1\le m\le n$, and let the subset $X\subset \R^n$ be defined by a system of equations and inequalities
$$
A\bx=\bb,\qquad \bx\ge0.
$$
A vector $\bx=(x^1,\ldots,x^n)\in\R^n$ is a corner point of a subset $X$ if and only if $x^j\ge0$ for all $j=1,\ldots,n$, there exist $m$ linearly independent columns $A_{j_1},\ldots,A_{j_m}$ of the matrix $A$ such that
$$
A_{j_1}x^{j_1}+\cdots+A_{j_m}x^{j_m}=\bb,
$$
and $x^i=0$ for all $i\not\in\{j_1,\ldots,j_m\}$.
\end{prop}

Using Proposition~\ref{prop:Lin_Prog_Vert}, explicit weight formulas for minimal parametric generalized fillings of type $G$ were written in~\cite{IT_LP} as the maximum of the objective function values at the vertices of the polytope $W_G$ for metric spaces consisting of at most $6$ points. Note that the cases $n\ge 4$ were described in~\cite{IT_MSb}, and the case $n=5$ is considered in~\cite{BBed}. Later, O.~Shcherbakov noted~\cite{Shch1} that the vertices of the polytope $W_G$ correspond exactly to irreducible multi-tours of the tree $G$, that is, the weight formulas from~\cite{Eremin} and~\cite{IT_LP} coincide. Moreover, in the same paper~\cite{Shch1} it is noted that the maximum can be attained at any vertex of the polytope $W_G$ (that is, it is possible to choose a suitable metric space); in other words, it is generally impossible to reduce the enumeration of vertex (of irreducible multi-tours). Therefore, it is of interest to study the set of all vertices of the polytopes $W_G$.

\section[Matrix of all $2$-cuts and its polytop]{Matrix of all $2$-cuts and its polytop and trees' polyhedra}
We will need the following construction. Let $M=\{m_1,\ldots,m_n\}$ be a finite set of $n$ elements, and, as above, denote by $K_n(M)$ the complete graph with vertex set $M$. Consider the matrix $B$ of all $2$-cuts of $K_n(M)$, that is, $B$ is a matrix of size $C\times N$, where $N=n(n-1)/2$ is the number of edges of $K_n(M)$, and $C=2^{n-1}-1$ is the number of partitions of the $n$-element set $M=\{m_1,\ldots,m_n\}$ into two non-empty disjoint subsets (such partitions are called \emph{$2$-cuts\/}). The rows of the matrix $B$ correspond to partitions of the set $M$, the columns correspond to edges $\{m_i,m_j\}$ of the graph $K_n(M)$, numbered, as above, by pairs $ij$. We denote the elements of the matrix $B$ by $b_{ij}^p$, where $ij$ is an edge of the graph $K_n(M)$, and $p$ is a partition of the set $M$. If $p=\{M_1,M_2\}$ is a partition of the set $M$, that is, $M=M_1\sqcup M_2$, then $b_{ij}^p=1$ if $m_i$ and $m_j$ belong to different subsets of the partition, that is, $\big|\{m_i,m_j\}\cap M_k\big|=1$, $k=1,\,2$, and $b_{ij}^p=0$ otherwise. Denote by $T_n\subset\R^N$ a convex polyhedral subset defined in $R^N$ by the following system of inequalities: $B\bx\ge\one_C$, $\bx\ge0$, where $\bx\in\R^N$.

\begin{rk}
For any binary tree $G$ connecting $M$, the matrix $A_G$ is a submatrix of the matrix $B$.
\end{rk}

\begin{ass}\label{ass:subset}
For an arbitrary binary tree $G$ with boundary $M$, $|M|=n$, the inclusion $W_G\subset T_n$ holds.
\end{ass}

\begin{proof}
Since $W_G$ and $T_n$ are convex, it suffices to verify that all vertices of $W_G$ lie in $T_n$. Let $\bv\in W_G$ be an arbitrary vertex of $W_G$. It corresponds to some irreducible multi-tour $\sigma$ of the tree $G$ of some multiplicity $k$; namely, the vector $2k\bv$ is a multi-tour vector, and therefore $\bv\ge0$ and $A_G\bv=\one_{2n-3}$, that is, the inequalities of the system $B\bv\ge\one_C$ corresponding to the rows of the matrix $A_G$, hold as equalities.

Now consider a row of the matrix $B$ that is not a row of the matrix $A_G$. This row corresponds to some partition $M=M_1\sqcup M_2$ of the set $M$. We need to show that at least $2k$ boundary paths of the multi-tour $\sigma$ connect vertices from different $M_1$ and $M_2$. We use induction on the number $n$ of boundary vertices in $G$. In the trivial case $n=2$, the statement is obvious. Now let $n\ge3$, and let some partition $M=M_1\sqcup M_2$ be fixed.

First, suppose that a tree $G$ has moustaches, both vertices $m$ and $m'$ of degree $1$ of which belong to the same subset, say $M_1$. We reconstruct the tree by discarding these moustaches (both the pair of vertices and the pair of edges), denote the resulting tree by $G'$, and the new vertex of degree $1$ by $m''$. The set $M'$ of vertices of degree $1$ of $G'$ has the form $M'=\{m''\}\cup M\setminus\{m,m'\}$. Consider its partition $M'=M'_1\sqcup M'_2$, where $M'_1=\{m''\}\cup M_1\setminus\{m,m'\}$ and $M'_2=M_2$. A multi-tour $\sigma$ of $G$ generates a multi-tour of $G'$ as follows: delete all boundary paths $\g(m,m')$, then in all boundary paths incident with $m$ or $m'$, replace these vertices with $m''$, leaving the remaining boundary paths unchanged. This yields a multi-tour $\sigma'$ of $G'$ of the same multiplicity $k$. By the induction hypothesis applied to $G'$ and $\sigma'$, the multi-tour $\sigma'$ has at least $2k$ boundary paths connecting vertices from different sets $M'_1$ and $M'_2$. If such a path is not incident with $m''$, then it is also a boundary path in the multi-tour $\sigma$ connecting the same vertices in $G$. Otherwise, to obtain a boundary path from the original multi-tour $\sigma$, we must replace $m''$ with $m'$ or $m$. In both cases, we obtain a boundary path from $\sigma$ with ends in different sets $M_1$ and $M_2$. Thus, in the case under consideration, the number of boundary paths in $\sigma$ connecting vertices from different sets is at least $2k$. Note that for $n=3$, one of the tree's moustaches always belongs to the same subset, i.e., the case $n=3$ is done.

Now suppose that $n\ge 4$, and the pairs of vertices of all moustaches of the tree $G$ belong to different subsets. But then for each pair $m$, $m'$ of such vertices, the multi-tour $\sigma$ includes $k$ boundary paths connecting $m$ and $m'$ (Proposition~\ref{prop:mult_tour_edges_pairs}). Since for $n\ge 4$ the tree has at least two pairs of disjoint moustaches, in this case the number of boundary paths in $\sigma$ connecting vertices from different subsets is at least $2k$. The statement is proved.
\end{proof}

\begin{ass}\label{ass:vertices}
Let $G$ be a binary tree with boundary $M$. Each vertex of $W_G$ is a vertex of $T_n$.
\end{ass}

\begin{proof}
Let us pass from the system of inequalities $B\bx\ge1$ to a system of equations by adding auxiliary non-negative variables $\by\in\R^C$ and passing to the extended block matrix $\hB=(B,-E_C)$, where $E_C$ denotes the identity matrix of size $C\times C$. Consider a convex polyhedral set $\hT_n\subset\R^{C+N}$ defined by the following system of equations and inequalities: $\hB(\bx,\by)=\one_{C}$, $\bx\ge0$, $\by\ge0$. The following result is widely used in linear programming when passing from General LP Problem to Classical LP Problem.

\begin{lem}\label{lem:clp_glp}
If $\bx\in\R^N$ belongs to the polyhedron $T_n$ defined by the system of inequalities $B\bx\ge\one_C$, $\bx\ge0$, then $(\bx,\by)$, where $\by=B\bx -\one_C$, belongs to $\hT_n$, that is, $\hB(\bx,\by)=\one_C$, $\bx\ge0$ and $\by\ge0$. Conversely, if $(\bx,\by)\in\R^{N+C}$ belongs to the polyhedron $\hT_n$ defined by the system of equations and inequalities $\hB(\bx,\by)=\one_C$, $\bx\ge0$ and $\by\ge0$, then $\by=B\bx -\one_C$ and $\bx\in T_n$, that is $B\bx\ge\one_C$, $\bx\ge0$. The projection map $\pi\:\hT_n\to T_n$, $(\bx,\by)\mapsto\bx$ is one-to-one, and the inverse map is $\bx\mapsto(\bx,B\bx-\one_C)$. Moreover, $\bx^*$ is a corner point of $T_n$ if and only if $(\bx^*,\by^*)$, $\by^*=B\bx^*-\one_C$, is a corner point of $\hT_n$.
\end{lem}

\begin{proof}
If $\bx\in T_n$, then $\bx\ge0$ and $\by=B\bx -\one_C\ge0$ by definition of the polyhedron $T_n$, therefore $\hB(\bx,\by)=B\bx-\by=\one_C$, whence $(\bx,\by)\in\hT_n$. Conversely, if $(\bx,\by)\in\hT_n$, then $\bx\ge0$, $\by\ge0$ and $\hB(\bx,\by)=B\bx-\by=\one_C$, whence $B\bx-\one_C=\by\ge0$, that is, $B\bx\ge\one_C$, and, therefore, $\bx\in T_n$.

Further, from the already proven first two statements of Lemma it follows that for any point $(\bx,\by)\in\hT_n$ the inclusion $\bx\in T_n$ holds, and conversely, $(\bx,B\bx-\one_C)\in\hT_n$ holds for each $\bx\in T_n$, and for any point $(\bx,\by)\in\hT_n$ the equality $\by=B\bx-\one_C$ holds, that is, $\by$ is uniquely determined by $\bx$, which means that the projection $\pi$ is bijective.

Now let $\bx^*$ be a corner point of $T_n$, and let $(\bx^*,\by^*)$, where $\by^*=B\bx^*-\one_C$, be not a corner point for $\hT_n$. The latter means that in $\hT_n$ there exist two distinct points $\bp_0=(\bx_0,\by_0)$ and $\bp_1=(\bx_1,\by_1)$ such that $(\bx^*,\by^*)=(1-t)\bp_0+t\bp_1$ for some $t\in(0,1)$. But then $\bx^*=(1-t)\bx_0+t\bx_1$ for the same value of $t$, and $\bx_0$ and $\bx_1$, as we have already proved, belong to $T_n$, therefore $\bx^*$ is not a corner point of $T_n$, a contradiction. Conversely, let $(\bx^*,\by^*)$ be a corner point of $\hT_n$, and let $\bx^*$ not be a corner point of $T_n$. Then there exist points $\bx_0$ and $\bx_1$ in $T_n$ for which $\bx^*=(1-t)\bx_0+t\bx_1$ for some $t\in(0,1)$. Let's put $\by_0=B\bx_0-\one_C$ and $\by_1=B\bx_1-\one_C$. Then $\by_0\ge0$ and $\hB(\bx_0,\by_0)=\one_C$, so $\bp_0=(\bx_0,\by_0)\in\hT_n$, and, in the same way, $\bp_1=(\bx_1,\by_1)\in\hT_n$. Besides,
\begin{multline*}
(1-t)\by_0+t\by_1=(1-t)(B\bx_0-\one_C)+t(B\bx_1-\one_C)=\\
=B\big((1-t)\bx_0+t\bx_1\big)-\one_C=B\bx^*-\one_C=\by^*,
\end{multline*}
from where $(\bx^*,\by^*)=(1-t)\bp_0+t\bp_1$ for the same value of $t\in(0,1)$, that is, $(\bx^*,\by^*)$ is not a corner point for $\hT_n$, a contradiction.
\end{proof}

Thus, let $\bv$ be a vertex of the polyhedron $W_G$. By Proposition~\ref{prop:Lin_Prog_Vert} this means that the matrix $A=A_G$ has a non-singular $(2n-3)\times(2n-3)$ submatrix $A_{j_1\ldots j_{2n-3}}$ generated by columns with numbers $j_1,\ldots,j_{2n-3}$, for which the equality $A_{j_1\ldots j_{2n-3}}\bv^{j_1\ldots j_{2n-3}}=\one_{2n-3}$ holds, where $\bv^{j_1\ldots j_k}\in\R^k$ denotes the vector formed by the components of the vector $\bv$ with numbers $j_1,\ldots,j_k$, and, moreover, all components of the vector $\bv$ are non-negative, and the components with numbers, different from $j_1,\ldots, j_{2n-3}$ are equal to zero.

Consider a polyhedral set $\hT_n$ and construct its vertex $(\bx,\by)$ corresponding to $\bv$. To do this, we set $\bx=\bv$, then $\bx\in T_n$ by Assertion~\ref{ass:subset}, therefore $B\bx\ge\one_C$. We set $\by=B\bx-\one_C\ge0$. Recall that the matrix $A_G$ is a submatrix of the matrix $B$, therefore the inequalities from the system $B\bx\ge\one_{C}$ corresponding to the rows of $A_G$ are satisfied in the form of equalities, and hence the corresponding components $\by$ are equal to zero. If necessary, we permute the rows of the matrix $\hB$ so that the first $(2n-3)$ rows of its submatrix $B$ become rows of the matrix $A_G$. Consider a $C\times C$ submatrix $\hB_C$ of $\hB$ generated by columns $j_1,\ldots,j_{2n-3}$ of its submatrix $B$ and $C-(2n-3)$ columns of the matrix $-E_C$ with numbers greater than $2n-3$ (we assume that the rows and columns of a square matrix $-E_C$ are numbered starting from $1$). This submatrix has a block form: the diagonal contains the non-singular matrix $A_{j_1\ldots j_{2n-3}}$ and the scalar matrix $-E_{C-(2n-3)}$, and the block above the diagonal consists of zeros. Therefore, $\hB_C$ is non-singular. It remains to note that all nonzero components of the constructed vector $(\bx,\by)$ correspond to columns of the matrix $\hB_C$, therefore $(\bx,\by)$ is a vertex of $\hT_n$ by Proposition~\ref{prop:Lin_Prog_Vert}, and, by Lemma~\ref{lem:clp_glp}, $\bx=\bv$ is a vertex of $T_n$. The statement is proved.
\end{proof}

So, we are ready to prove the following Theorem.

\begin{thh}\label{thm:all_conv}
Consider the family of all binary trees connecting a finite set $M$, $|M|=n$. For each such tree $G$, construct a polyhedron $W_G$, denote by $V_G\subset\R^N$ the set of its vertices, and let $V=\cup_GV_G\subset\R^N$. Then the set $V$ is extremal, that is, $V$ is the set of vertices of its convex hull $\conv V\subset\R^N$.
\end{thh}

\begin{proof}
By Assertion~\ref{ass:subset}, every set $V_G$, and hence the entire set $V$, is contained in a convex polyhedral set $T_n$, so the convex hull of $\conv V$ is contained in $T_n$. By Proposition~\ref{ass:vertices}, every point in $V$ is also a vertex for $T_n$, so $v$ is also a vertex for $\conv V$, as required.
\end{proof}

\subsection*{Acknowledgements}
The work of A.\,O.~Ivanov was partly supported by Moscow Mathematical Center for Fundamental and Applied Mathematics (Agreement No.~075--15--2025--345) and by Sino-Russian Mathematical Center at Peking University.

The work of A.\,A.~Tuzhilin was supported by grant No.~25--21--00152 of the Russian Science Foundation, by National Key R\&D Program of China (Grant No.~2020YFE0204200), as well as by the Sino-Russian Mathematical Center at Peking University. Partial of the work by A.\,A.~Tuzhilin were done in Sino-Russian Math. center, and he thanks the Math. Center for the invitation and the hospitality.


\end{document}